\documentclass{amsart}
\usepackage{amsmath}
\usepackage{amssymb}
\usepackage{mathrsfs}
  \usepackage{paralist}
  \usepackage{graphics} %% add this and next lines if pictures should be in esp format
 \usepackage[colorlinks=true]{hyperref}
\newtheorem{theorem}{Theorem}[section]
\newtheorem{corollary}[theorem]{Corollary}

\newtheorem{lemma}[theorem]{Lemma}

\newtheorem*{conjecture}{Conjecture}

\theoremstyle{definition}

\newtheorem*{remark}{Remark}

\newcommand{\scru}{\mathscr{U}}
\newcommand{\scrv}{\mathscr{V}}
\newcommand{\scrw}{\mathscr{W}}
\newcommand{\scrg}{\mathscr{G}}
\newcommand{\scrh}{\mathscr{H}}
\newcommand{\scrf}{\mathscr{F}}
\newcommand{\scrl}{\mathscr{L}}
\newcommand{\al}{\alpha}
\newcommand{\osc}{\mathrm{osc}}

\title[A Generalization of the Gauss-Kuzmin-Wirsing constant]
      {A Generalization of the Gauss-Kuzmin-Wirsing constant}

\author[Peng Sun]{}


 \email{sunpeng@cufe.edu.cn}

\begin{document}

\maketitle\ 

% Enter the first author's name and address:
\centerline{\scshape Peng Sun}
\medskip
{\footnotesize
% please put the address of the first author
 \centerline{China Economics and Management Academy}
   \centerline{Central University of Finance and Economics}
   \centerline{Beijing 100081, China}
} % Do not forget to end the {\footnotesize by the sign }

\bigskip

%+Abstract
\begin{abstract}
 
We generalize the result of Wirsing on Gauss transformation to the generalized tranformation
 $T_p(x)=\{\cfrac{p}{x}\}$ for any positive integer $p$. 
%We show that for every absolutely continuous measure $\nu$,
% $$\nu(T^{-n}([0,x]))-\frac{\ln(p+x)-\ln p}{\ln(p+1)-\ln p}$$
We give an estimate for the generalized Gauss-Kuzmin-Wirsing constant.

\end{abstract}
%-Abstract

%+Contents
%\tableofcontents
%-Contents

\section{Introduction and Results}

%\vspace{20}
\bigskip

\bigskip

%\vskip{20}

Let $p$ be a positive integer. Consider the following generalized Gauss transformation
on $I=[0,1]$:
$$T(x)=T_p(x)=\begin{cases}0,& x=0,\\
\{\cfrac p x\},& x\ne 0,
\end{cases}$$
where  $\{x\}$ is the fractional part of $x$. $T_1$ is just the regular Gauss
transformation. It is intriguing to see how the whole
theory changes from $1$ to $p$.
From \cite{Sun1} we know that for every $p$,  $T_p$ has a unique absolutely continuous
ergodic invariant measure
$$d\mu_p(x)=\frac1{\ln (p+1)-\ln p}\cdot\frac 1{p+x}d\omega,$$
where $\omega$ is the Lebesuge measure on $I$. 
Equivalently, the density function $$\eta_p(x)=\frac1{\ln (p+1)-\ln p}\cdot\cfrac{1}{p+x}$$
is the unique continuous eigenfunction, corresponding to the eigenvalue $1$,
of the following generalized
Gauss-Kuzmin-Wirsing operator
$$(\mathscr{G}_pf)(x)=\sum_{T_p(y)=x}\frac{f(y)}{|T_p'(y)|}=\sum_{k=p}^\infty\frac{p}{(k+x)^2}f(\frac{p}{k+x}).$$

%$T=T_1$ is the regular Gauss transfermation. 
We denote
$$
%\varphi_n(x)=
\varphi_{p,n}(x)=\omega(T_p^{-n}([0,x])).$$
%For $p=1$,
and
$$\Phi_p(x)=\mu_p([0,x])=\frac{\ln(p+x)-\ln p}{\ln(p+1)-\ln p}.$$ 
Then 
$$\lim_{n\to\infty}\varphi_{p,n}(x)=\Phi_p(x).$$
%where we denote
In 1812, Gauss proposed the problem to compute
$$\Delta_{n}(x)=\varphi_{1,n}(x)-\Phi_1(x).$$
In 1928, Kuzmin \cite{Kuzmin} showed that
$$\Delta_{n}(x)=O(q^{\sqrt n})$$
as $n\to\infty$ for some $q\in(0,1)$. In 1929 L\'evy \cite{Levy} obtained
$$\Delta_{n}(x)=O(q^n)$$
for $q=3.5-2\sqrt2<0.7$.
In 1974 Wirsing \cite{Wir} established the estimate:
\begin{equation}\label{eqwir}
\Delta_n(x)=(-\lambda)^n\Theta(x)+O(x(1-x)\tau^n),
\end{equation}
where $\lambda=0.303663\cdots$ is the Gauss-Kuzmin-Wirsing constant, $0<\tau<\lambda$
and $\Theta(x)\in C^2(I)$ with $\Theta(0)=\Theta(1)=0$.

In \cite{Sun2} we have generalized L\'evy's result for $T_p$:
\begin{theorem}\label{kuzgen}
For every positive integer $p$ and every $x\in I$,
\begin{equation*}
%\label{Phiest}
\Delta_{p,n}(x)=\varphi_{p,n}(x)-\Phi_p(x)=O(Q_p^n)
\end{equation*}
with 
$$Q_p=2p^2\zeta(3,p)-p\zeta(2,p)<\frac{1}{2p}+\frac{3}{8p^2},$$
where
$$\zeta(2,p)=\sum_{k=p}^\infty\frac{1}{k^2},\;
%\text{ and }
\zeta(3,p)=\sum_{k=p}^\infty\frac{1}{k^3}$$
are the Hurwitz zeta functions.
\end{theorem}

In this paper we would like to generalize Wirsing's result for better estimates of $\Delta_{p,n}$.
Following Wirsing's approach (also referring to
\cite{IK}) we know that (\ref{eqwir}) can be generalized with the constant
 $\lambda=\lambda_p$ for every $p$. Moreover, we have a quite good estimate for $\lambda_p$.
It turns out that the asymptotic behavior of $\lambda_p$, as $p\to\infty$,
is not so far from that
of $Q_p$:
\begin{theorem}\label{thmain}
For every positive integer $p$ and every $x\in I$,
$$\Delta_{p,n}(x)=(-\lambda_p)^n\Theta(x)+O(x(1-x)\tau_p^n),$$
where $\Theta=\Theta_p\in C^2[0,1]$ with $\Theta(0)=\Theta(1)=0$,
$\lambda_p,\tau_p$ are contants such that $$0<\tau_p<\lambda_p.$$ 
The costant $\lambda_p$ satisfies the estimate:
$$v_p<\lambda_p<w_p,$$ where
$$v_p=\frac{p}{2(p^2+\frac{2}{3}p+\frac19)}, w_p=\frac{p}{2(p^2+\frac23p-\frac29)},$$
which yields
$$\lambda_p=\frac{1}{2p}-\frac{1}{3p^2}+O(\frac{1}{p^3}).$$
\end{theorem}

We actually show a slightly stronger result:
\begin{theorem}\label{gthmain}
For every positive integer $p$, there is a function $\Psi_p\in C^2(I)$ with
$\Psi_p(0)=\Psi_p(1)=0$, 
$\psi_p=((p+x)\Psi_p')'$ and a positive
bounded linear functional $\scrl_p$ on $C(I)$ such that the following
holds:

%Let $\nu$ be an absolutely continuous measure on $I$
For every $f_0\in C^1(I)$ such that $\int_If_0d\omega=1$ and the measure $d\nu=f_0d\omega$ on $I$, for every $n$ and
every $x\in I$,
\begin{align*}
&|\nu(T_p^{-n}([0,x])-\Phi_p(x)-(-\lambda_p)^n\scrl_p(g_0')\Psi_p(x)|\\
\le&\tau_p^n(\ln(p+1)-\ln
p)^2(\frac{p+1}{2})\|\psi_p\|\osc(\frac{g_0'}{\psi_p})\Phi_p(x)(1-\Phi_p(x)),
\end{align*}
where $\lambda_p$, $\tau_p$ are the constants in Theorem \ref{thmain},
$g_0=(p+x)f_0$, $\|\cdot\|$ is the
supremum norm on $C(I)$, and
$$\osc(f)=\sup f-\inf f$$
is the oscillation of $f$ on I.
\end{theorem} %gthmain

%\begin{remark}
%Note that the factor $$(\ln(p+1)-\ln p)^2(\frac{p+1}{2})<\frac1p.$$
%\end{remark}

For convenience and accuracy of estimates, from now on we will assume that $p\ge 2$.
Known results \cite{Wir}\cite{IK} has justified that the above theorems
 hold for $p=1$.

\section{The Operators} 
We fix $p\ge 2$. 
%Denote $$\varphi_n(x)=\varphi_{p,n}(x).$$
%, $\varphi_n(0)=0$ for all $n$
%Then 
Let
$$\varphi_0(x)=\nu([0,x])=\int_0^x f_0 dx,$$
where $f_0$ and $\nu$ are as in Theorem \ref{gthmain}.
Note
\begin{align*}
\varphi_{n+1}(x)
%=&\int_0^x T^{-(n+1)}(t)dt=\int_0^x T^{-n}(T^{-1}(t))dt\\
%=&\sum_{k=1}^\infty\int_{\frac{1}{k+x}}^{\frac{1}{k}} 
=&\nu(T^{-(n+1)}([0,x]))=\nu(T^{-n}(T^{-1}([0,x])))\\
=&\nu(T^{-n}(\bigcup_{k=p}^\infty [\frac{p}{k+x},\frac{p}{k}]))
=\sum_{k=p}^\infty(\varphi_n(\frac{p}{k})-\varphi_n(\frac{p}{k+x})).
\end{align*}
This recursive formula implies that $\varphi_n\in C^2(I)$ for all $n$ 
and hence
$$\varphi_{n+1}'(x)=\sum_{k=p}^{\infty}\frac{p}{(k+x)^2}\varphi_n'(\frac{p}{k+x})=(\mathscr{G}_p\varphi_n')(x).$$
 Let 
 %\begin{equation}\label{subgn}
 %(\mathscr{U}g)(x)=(p+x)(\mathscr{G}_pg)(x).
 %\end{equation}
$$g_n(x)=(p+x)\varphi_n'(x)$$
(Note that $g_0=(p+x)\varphi_0'=(p+x)f_0$ as in Theorem \ref{gthmain}).
Then
$$g_{n+1}(x)=(\scru g_n)(x),$$
where
%$$\frac{g_{n+1}(x)}{p+x}=(\mathscr{G}_p^n())(x).$$
$$(\mathscr{U}g)(x)=\sum_{k=p}^\infty g(\frac{p}{k+x})h_k(x)$$
%(Note that $\lim_{k\to\infty}g(\cfrac{p}{k+1+x})$)
and
$$h_k(x)=\frac{p+x}{(k+x)(k+1+x)}=\frac{k+1-p}{k+1+x}-\frac{k-p}{k+x}.$$
For $g\in C^1(I)$,
$$(\scru g)(x)=g(0)+
\sum_{k=p}^\infty\frac{k+1-p}{k+1+x}(g(\frac{p}{k+x})-g(\frac{p}{k+1+x})),$$
Hence we have
\begin{align*}
(\scru g)'(x)=&-\sum_{k=p}^\infty (\frac{k+1-p}{(k+1+x)^2}(g(\frac{p}{k+x})-g(\frac{p}{k+1+x}))
\\&+\frac{p}{(k+x)^2}g'(\frac{p}{k+x})h_k(x))
%=&\sum_{k=p}^\infty(g_n(\frac{p}{k+x})-g_n(\frac{p}{p+x}))h_k'(x)-\sum_{k=p}^\infty
%\frac{p}{(k+x)^2}g_n'(\frac{p}{k+x})h_k(x)\\
%=&-\sum_{k=p}^\infty \frac{p(k-p)}{(p+x)(k+x)}g_n'(\frac{p}{\tau_k+x})h_k'(x)-\sum_{k=p}^\infty
%\frac{p}{(k+x)^2}g_n'(\frac{p}{k+x})h_k(x)\\
=-\scrv g',
\end{align*}
where
\begin{align*}
\scrv f(x)=\sum_{k=p}^\infty(\frac{k+1-p}{(k+1+x)^2}\int^{\frac{p}{k+x}}_{\frac{p}{k+1+x}}f(y)dy+\frac{ph_k(x)}{(k+x)^2}f(\frac{p}{k+x})).
\end{align*}
Then
$$(\scru^n g)'=(-1)^n\scrv^n g'$$
for all $g\in C^1(I)$ and all $n$. We will then concentrate on the operator
$\scrv$. 
As the crucial part of this paper, we find a auxiliary function to estimate
the eigenvalue of $\scrv$:
%Firstly, we need an auxiliary function.  The following
%the theorem is the crucial part of this paper. 

\begin{theorem}\label{auxfun}
There is $\xi\in C(I)$ and $\xi>0$ such that
\begin{equation}\label{estv}
v_p\xi\le\scrv\xi\le w_p\xi.
\end{equation}
\end{theorem} %propmain

\begin{remark}For $p\ge 2$, we may explicitly choose
$$\xi(x)
%=\Xi_p(x)
=\frac{1}{(p+\frac13+x)^2}.$$
\end{remark}

Linearity and positiveness of $\scrv$ leads to the following estimate:
\begin{corollary}\label{genest}
Given $h\in C(I)$ and $h>0$. Denote
$$\beta_1=\min_{x\in I}\frac{\xi(x)}{h(x)},\beta_2=\max_{x\in I}\frac{\xi(x)}{h(x)}.$$
Then for every $n$,
$$\frac{\beta_1}{\beta_2}v_p^nh\le\scrv^nh\le\frac{\beta_2}{\beta_1}w_p^nh.$$
\end{corollary}

The estimate (\ref{estv}) implies that $\scrv$ has an eigenvalue $\lambda=\lambda_p$,
which is the largest one in absolute value,
between $v_p$ and $w_p$:
\begin{theorem}\label{thmeigen}
Assume that there is $\xi\in C(I)$ and $\xi>0$ such that
$$v_p\xi\le\scrv\xi\le w_p\xi.$$
Then $\scrv$ has an eigenvalue $\lambda=\lambda_p\in[v_p,w_p]$ with a positive
eigenfunction $\psi=\psi_p$ and $\psi\ge\xi$. Moreover, for every $f\in C(I)$,
$$\scrv^n f=\lambda^n\scrl(f)\psi+\tau^n\osc(\frac{f}{\psi})\kappa_n\psi,$$
where $\scrl:C(I)\to\mathbb{R}$ is a positive bounded linear functional with
$\|\scrl\|\le \cfrac1{m(\xi)}$, $\tau$ is a constant with
$0<\tau<\lambda$ and $\kappa_n:I\to\mathbb{R}$ is a function
with $|\kappa_n|\le 1$.
\end{theorem}

Theorem \ref{thmeigen} is a consequence of the following theorem of
Wirsing:
% on the operators like $\scrv$:
%. We state it here for completeness:

\begin{theorem}[cf. \cite{Wir} 
{\cite[Theorem 2.2.4]{IK}}]\label{Wirthm}
 Let $\scrw:C(I)\to C(I)$ be a positive bounded linear
operator and $\scrf:C(I)\to\mathbb{R}$ a positive bounded linear functional
such that $\scrw\ge\scrf$. Assume that there is $\theta\in C(I)$ and two constants
$v,w$, $0<v\le w$, such that
$$m(\theta)=\inf_{x\in I}\theta(x)>0,$$
$$v\le\frac{\scrw\theta(x)}{\theta(x)}\le w$$
for all $x\in I$ and
\begin{equation}\label{estf}
\scrf(\theta)>(1-\frac{v}{w})\|\scrw\theta\|.
\end{equation}
Then $\scrw$ has an eigenvalue $\lambda\in[v,w]$ with a positive eigenfunction
$\tilde\theta$ such that
$$\tilde\theta\ge\theta\ge m(\theta)>0,$$
$$0<w\frac{\scrf(\theta)}{\|\scrw\theta\|}-(w-v)\le\frac{\scrf\tilde\theta}{\|\tilde\theta\|}\le\lambda,$$
and for every $f\in C(I)$ and every $n$ we have
$$\scrw f=\lambda^n\scrl(f)\tilde\theta+(\lambda-\frac{\scrf\tilde\theta}{\|\tilde\theta\|})^n\osc(\frac{f}{\tilde\theta})\kappa_n\tilde\theta,$$
where $\scrl:C(I)\to\mathbb{R}$ is a positive bounded linear functional with
$\|\scrl\|\le\cfrac1{m(\theta)}$ and $\kappa_n:I\to\mathbb{R}$ is a function
with $|\kappa_n|\le 1$.
\end{theorem}

\section{Proof of Theorem \ref{auxfun}}
%\begin{proof}
Fix $p\ge 2$. Let $H$ be a bounded continuous function on $\mathbb{R}^+$ such that
$$\lim_{x\to\infty}\frac{H(x)}{x}=0.$$
We look for $g$ on $(0,1]$ such that $\scru g=H$. 
Assume that
$$\scru g(x)=\sum_{k=p}^\infty\frac{p+x}{(k+x)(k+1+x)}g(\frac{p}{k+x})=H(x)$$
for all $x\in\mathbb{R}^+$. Then
$$\frac{H(x)}{p+x}-\frac{H(1+x)}{p+1+x}=\frac{1}{(p+x)(p+1+x)}g(\frac{p}{p+x}),
x\in\mathbb{R}^+.$$
Hence
$$g(y)=(\frac{p}{y}+1)H(\frac{p}{y}-p)-\frac{p}{y}H(\frac{p}{y}-p+1), y\in(0,1].$$
It is easy to check that such $g$ indeed satisfies $\scru g=H$.

For $a\in I$, define
$$H_a(x)=\frac{1}{p+a+x}.$$
Then
$$g_a(x)=(\frac{p}{x}+1)H_a(\frac{p}{x}-p)-\frac{p}{x}H_a(\frac{p}{x}-p+1)=\frac{p+x}{p+ax}-\frac{p}{p+(1+a)x}$$
satisfies $\scru g_a(x)=H_a(x)$ for all $x\in I$.
Let
$$\xi_a(x)=g_a'(x)=\frac{p(1-a)}{(p+ax)^2}+\frac{p(1+a)}{(p+(1+a)x)^2}>0.$$
Then
$$\scrv\xi_a(x)=-(\scru g_a)'(x)=\frac{1}{(p+a+x)^2}.$$
We would like to choose $a$ such that
$$\frac{\xi_a}{\scrv\xi_a}(0)=\frac{\xi_a}{\scrv\xi_a}(1).$$
So we must have
$$\frac2p(p+a)^2=(\frac{p(1-a)}{(p+a)^2}+\frac{p(1+a)}{(p+1+a)^2})(p+1+a)^2,$$
which yields
\begin{equation}\label{rhoo}
2(p+a)^4-(2p^3+p^2)(p+a)-p^2(p+1)=0.\end{equation}
Denote the left-hand side of (\ref{rhoo}) by $\rho(a)$. As $$\rho'(a)=8(p+a)^3-(2p^3+p^2)>0$$ for $a\ge 0$, 
$\rho(a)=0$ has a unique positive real root
$$a=\alpha=\alpha_p.$$
We note that
$$\rho(a)=(6a-2)p^3+(12a^2-a-1)p^2+8a^3p+2a^4.$$
Hence for $p\ge 2$,
$$\rho(0.32)=-0.08 p^3- 0.0912 p^2  + 0.262144 p + 0.02097152<0$$
and $$\rho(\frac13)=\frac{8p}{27}+\frac{2}{81}>0.$$
So $\alpha\in(0.32, \cfrac13)$. We also have $$\lim_{p\to\infty}\alpha_p=\frac13.$$
For $a\in[\al,\cfrac13]$, $\cfrac{\xi_{a}}{\scrv\xi_a}$ attains its maximum
$$M(a)=\frac2p(p+a)^2\le\frac{1}{v_p}.
%>M(\frac13)=\frac2p(p+\frac13)^2
$$
at $x=0$, and has a minimum (see Lemma \ref{computma})
\begin{align}\label{eqma}
m(a)=&\frac1p((1-a)p^2+a^2(1+a)+\\
&3(p-pa-a^2)(pa+a+a^2)((1-a)\gamma+(1+a)\gamma^{-1}))\nonumber
\end{align}
at 
$$x=x_0=\frac{(\gamma-1)p}{1+a-a\gamma}\in I$$
with
$$\gamma=\gamma_a=\frac{p+(1+a)x_0}{p+a x_0}=\sqrt[3]{\frac{(1+a)(pa+a+a^2)}{(1-a)(p-pa-a^2)}}<\frac{p+1+a}{p+a}.$$
Note that
$$\gamma_\al^3=1+\frac{3\al^2+(3p+1)\al-p}{(1-\al)(p-p\al-\al^2)}\ge1.$$
We have
$$\al\ge\frac{\sqrt{9p^2+18p+1}-(3p+1)}{6}=\frac{2p}{\sqrt{9p^2+18p+1}+(3p+1)}>\frac{p}{3p+2}.$$
For $a\in[\al,\cfrac13]$
%,$$\gamma_a\le\sqrt[3]{1+\frac{3a^2+a}{\frac4{9}p-\frac2{27}}}<$$
 and $p\ge 2$,
$$\gamma\le\gamma_{\frac13}=\sqrt[3]{1+\frac1{\frac23p-\frac19}}<1+\frac{1}{2p
%-\frac13
}<1+a.$$
Then
$$(1-a)\gamma+(1+a)\gamma^{-1}\ge2+\frac{1-a}{2p}-\frac{1+a}{2p+1}>2-\frac{a}{p},$$
as the left-hand side is decreasing in $\gamma$ for $1<\gamma<1+a$.
So
\begin{equation}\label{m11}
m(a)>\frac1p\big((1-a)p^2+a^2(1+a)+3(p-pa-a^2)(pa+a+a^2)(2-\frac{a}{p})\big)
\end{equation}
%Denote the right-hand side of (\ref{m11}) by $m_1(a)$. Then $m_1(a)$ is
%decreasing for $a>0.32$. 
Therefore
$$m(\frac13)>\frac{2}{p}(p^2+\frac23p-\frac{11}{54}+\frac{2}{81p})>\frac{1}{w_p}
%\frac{2}{p}(p^2+\frac23p-\frac{2}{9})
.$$

In conclusion, we have
$$m(\frac13)\le\frac{\xi_\frac13}{\scrv\xi_\frac13}\le M(\frac13).$$
For convenience, we choose
\begin{equation}\label{xi31}
\xi=\scrv\xi_{\frac13}=\frac{1}{(p+\frac13+x)^2}.
\end{equation}
Then 
%for $a\in[\al,\frac13]$,
$$v_p\xi\le\scrv\xi\le w_p.$$

\section{The Eigenvalue}\label{pfeigen}
In this section we apply Theorem \ref{Wirthm} to show Theorem \ref{thmeigen},
i.e. $\scrv$ has an
eigenvalue in $[v_p,w_p]$.

Fix $p\ge 2$. To apply Theorem \ref{Wirthm}, put $\scrw=\scrv$, $v=v_p$, $w=w_p$ and $\theta=\xi$ as in (\ref{xi31}). We just need to identify a linear functional $\scrf$
%as in Theorem \ref{Wirthm}
.
For every $f\in C(I)$ and $f\ge 0$,
\begin{align*}
\scrv f(x)\ge\sum_{k=p}^\infty\frac{k+1-p}{(k+1+x)^2}\int^{\frac{p}{k+x}}_{\frac{p}{k+1+x}}f(y)dy
=\int_{0}^1 k(x,\lfloor\frac{p}{y}-x\rfloor)f(y)dy,
\end{align*}
where
$$k(x,0)=0, x\in I,$$
%$$k(x,y)=\frac{\lfloor\frac{p}{y}-x\rfloor+1-p}{(x+\lfloor\frac{p}{y}-x\rfloor+1)^2}.$$
$$k(x,t)=\frac{t+1-p}{(t+1+x)^2},$$
and $\lfloor x\rfloor$ denotes the largest integer no more than $x$.
Note that for fixed $x\in I$, for $t\ge 0$, $k(x,t)$ is (with $t$) increasing for $t-x\le 2p-1$ and 
decreasing for $t-x\ge 2p-1$.

For $0<y\le\frac{p}{2p+1}$, we have $\lfloor\frac{p}{y}-x\rfloor-x\ge2p-1$.
Hence
%decreasing for $x\in I$ and $t\ge 2p$. So in this case,
$$k(x,\lfloor\frac{p}{y}-x\rfloor)\ge\frac{\frac{p}{y}-x+1-p}{(\frac{p}{y}+1)^2}\ge\frac{py(1-y)}{(p+y)^2}
%\ge\frac{y}{2(p+y)}.
$$
For $\frac{p}{2p+1}<y\le\frac{1}{2}$, we have $2p-1\le\lfloor\frac{p}{y}-x\rfloor\le2p$.
Hence\begin{align*}k(x,\lfloor\frac{p}{y}-x\rfloor)\ge
%\frac{(2p-1+x)+1-p}{((2p-1+x)+1+x)^2}\ge\frac{1}{2(p+x)}\ge\frac{1}{2(p+1)}.
&\min\{k(x,2p-1), k(x,2p)\}\\
=&\min\{\frac{p}{(2p+x)^2},\frac{p+1}{(2p+1+x)^2}\}\\
\ge&\frac{p}{(2p+1)^2}.
\end{align*}
For $\frac12<y\le\frac{p}{p+1}$, we have $\lfloor\frac{p}{y}-x\rfloor-x\le2p-1$.
Hence
$$k(x,\lfloor\frac{p}{y}-x\rfloor)\ge\frac{\frac{p}{y}-x-p}{(\frac{p}{y})^2}\ge\frac{y(p-(p+1)y)}{p^2}.$$

We define for all $f$,
\begin{align*}\scrf(f)
=&\int_{0}^{\frac{p}{2p+1}}\frac{py(1-y)}{(p+y)^2}f(y)dy
+\int_{\frac{p}{2p+1}}^{\frac12}\frac{p}{(2p+1)^{2}}f(y)dy\\
&+\int_{\frac12}^{\frac{p}{p+1}}\frac{y(p-(p+1)y)}{p^2}f(y)dy.
\end{align*}
Then $\scrf$ is positive and $\scrf(f)\le\scrv f$ for all positive $f$.
We have
\begin{align*}
\scrf(\xi)
=&\int_{0}^{\frac{p}{2p+1}}\frac{py(1-y)}{(p+y)^2}\frac{dy}{(p+\frac13+y)^2}
%=&\int_{0}^{\frac{p}{2p+1}}\frac{y}{2(p+y)}\frac{dy}{(p+\frac13+y)^2}
+\int_{\frac{p}{2p+1}}^{\frac12}\frac{p}{(2p+1)^2}\frac{dy}{(p+\frac13+y)^2}\\
&+\int_{\frac12}^{\frac{p}{p+1}}\frac{y(p-(p+1)y)}{p^2}\frac{dy}{(p+\frac13+y)^2}\\
\ge&\frac{p(1-\frac{p}{2p+1})}{(p+\frac{p}{2p+1})^2(p+\frac13+\frac12)^2}\int_{0}^{\frac{p}{2p+1}}ydy
+\frac{p(\frac12-\frac{p}{2p+1})}{(2p+1)^{2}(p+\frac13+\frac12)^2}\\
&+\frac{\frac12}{p^2(p+\frac13+\frac12)^2}\int_{\frac12}^{\frac{p}{p+1}}(p-(p+1)y)dy\\
=&\frac{1}{16(p+\frac56)^2}(\frac{p}{(p+1)(p+\frac12)}+\frac{2p}{(p+\frac12)^3}+\frac{(p-1)^2}{p^2(p+1)})\\
>&\frac{1}{16(p+\frac56)^2}(\frac{p-\frac12}{(p+\frac12)^2}+\frac{2-\frac1p}{(p+\frac12)^2}+\frac{p^2-2p+1}{p(p+\frac12)^2})\\
=&\frac{p-\frac14}{8(p+\frac56)^2(p+\frac12)^2}.
\end{align*}
For $p\ge 2$,
\begin{align*}
\frac{w_p\scrf(\xi)}{\|\scrv\xi\|}\ge&\frac{\scrf(\xi)}{\|\xi\|}=(p+\frac13)^2\scrf(\xi)=\frac{(p+\frac13)^2}{(p+\frac56)^2}\cdot\frac{p-\frac14}{8(p+\frac12)^2}
>\frac{p^2(p-\frac14)}{8(p+\frac12)^4}\\
=&\frac{p^2(p-\frac14)}{8(p^2+\frac23p+\frac19)(p^2+\frac23p-\frac29)}\cdot\frac{(p+\frac13)^4}{(p+\frac12)^4}
\cdot\frac{p^2+\frac23p-\frac29}{(p+\frac13)^2}\\
\ge&(w_p-v_p)\cdot\frac34p(p-\frac14)(\frac{14}{15})^4(\frac{46}{49})\\
>&(w_p-v_p) \frac p2(p-\frac14)>w_p-v_p.
\end{align*}
%\end{align*}
Thus (\ref{estf}) holds for $\theta=\xi$. Hence Theorem \ref{thmeigen} follows
from Theorem \ref{Wirthm} with
$\psi=\tilde\theta$ and 
\begin{equation}\label{esttau}
\tau=\lambda-\frac{\scrf\psi}{\|\psi\|}\le\lambda-(\frac{w_p\scrf(\xi)}{\|\scrv\xi\|}-(w_p-v_p)).\end{equation}
%\end{proof}
\begin{remark} From (\ref{esttau}) we have
\begin{align}
\tau<&\lambda-(w_p-v_p)(\frac{p(p-\frac14)}{2}-1)\nonumber\\
=&\lambda-\frac{p}{2(p^2+\frac23p-\frac29)}\cdot\frac{\frac13(p^2-\frac14 p-2)}{2(p+\frac13)^2}\nonumber\\
\label{esttau2}<&\lambda-\frac{9}{198}w_p\le\frac{189}{198}\lambda.
%=&v_p\cdot\frac{\frac23p^2+\frac78p+\frac49}{p^2+\frac23p-\frac29}<\frac{175}{184}\lambda
\end{align}
\end{remark}

\section{Proof of the Theorem \ref{gthmain}}
Fix $p$. Let $\psi=\psi_p$ be as in Theorem \ref{thmeigen}, i.e. the eigenfunction
of $\scrv$ corresponding to $\lambda=\lambda_p$. Let 
$$\tilde\psi(x)=\int_0^x\psi_p(t)dt, x\in I$$
and
$$\Psi_p(x)=\int_0^x\frac{\tilde\psi(t)-\scru^\infty\tilde\psi}{p+t}dt, x\in  I.$$
Then
$$((p+x)\Psi_p'(x))'=\psi_p(x)$$
and $\Psi_p(0)=\Psi_p(1)=0$.

Let $\scrl_p=\scrl$ as in Theorem \ref{thmeigen} and 
%for 
$g_0=(p+x)f_0$.
For $n\in\mathbb{N}$ and $y\in I$ we set
\begin{align*}
D_n(y)=&\nu(T_p^{-n}([0,p(e^{y(\ln(p+1)-\ln p)}-1)]))-y\\
&-(-\lambda_p)^n\scrl_p(g_0')\Psi_p(p(e^{y(\ln(p+1)-\ln p)}-1)),
%y\in I,
\end{align*}
so that
$$D_n(\Phi_p(x))=\varphi_{n}(x)-\Phi_p(x)-(-\lambda_p)^n\scrl_p(g_0')\Psi_p(x),
x\in I.$$
Then 
\begin{align*}
\frac{1}{(\ln(p+1)-\ln p)^2}\cdot\frac{D_n''(\Phi_p(x))}{p+x}
=&(\scru^ng_0)'(x)-(-\lambda_p)^n\scrl_p(g_0')((p+x)\Psi_p'(x))'\\
=&(-1)^n\scrv^n g_0'(x)-(-\lambda_p)^n\scrl_p(g_0')\psi_p(x).
\end{align*}
By Theorem \ref{thmeigen},
$$|D_n''(\Phi_p(x))|\le\tau_p^n(\ln(p+1)-\ln
p)^2(p+1)\|\psi_p\|\osc(\frac{g_0'}{\psi_p}).$$
Note that $D_n(0)=D_n(1)=0$. We can apply the interpolation formula (see
Lemma \ref{lemintp}): for every $y\in I$ there is $t_y\in I$ such
that
\begin{equation}\label{eqintpap}
D_n(y)=-\frac{y(1-y)}{2}D_n''(t_y).
\end{equation}
So for every $n$ and every $x\in I$,
$$|D_n(\Phi_p(x))|\le\tau_p^n(\ln(p+1)-\ln
p)^2(\frac{p+1}2)\|\psi_p\|\osc(\frac{g_0'}{\psi_p})\Phi_p(x)(1-\Phi_p(x)).$$
Theorem \ref{gthmain} holds.

\section{Complementary Lemmas and Final Remarks}

%% Consider the auxillary function $G(t)=y(1-y)f(t)-t(1-t)f(y)$. $G(0)=G(1)=G(y)=0$.

We first show the computation of $m(a)$ in (\ref{eqma}) and the interpolation formula (\ref{eqintpap}).

\begin{lemma}\label{computma}
(\ref{eqma}) holds.
\end{lemma}

\begin{proof}
Denote
$$A=p-pa-a^2, B=pa+a+a^2.$$
Then
\begin{align*}
\frac{p+a+x_{0}}{p+ax_{0}}=\frac{p+a+\frac{(\gamma-1)p}{1+a-a\gamma}}{p+a\frac{(\gamma-1)p}{1+a-a\gamma}}=\frac{A\gamma+B}{p},
\end{align*}
\begin{align*}
\frac{p+a+x_0}{p+(a+1)x_0}=\frac{1}{\gamma}\frac{p+a+x_0}{p+ax_0}
=\frac{A+B\gamma^{-1}}{p}.
\end{align*}
Note $$(1-a)A\gamma^2=(1+a)B\gamma^{-1}.$$
Hence
\begin{align*}
m(a)=&\frac{\xi_a}{\scrv\xi_a}(x_0)\\
=&p(p+a+x_0)^2(\frac{1-a}{(p+ax_0)^2}+\frac{1+a}{(p+(1+a)x_0)^2})
%\\=&p(\frac{(1-a)(p+a+x)^2}{(p+ax)^2}+\frac{(1+a)(p+a+x)^2}{(p+(1+a)x)^2})
\\=&\frac1p((1-a){(A\gamma+B)^2}+(1+a)(A+B\gamma^{-1})^2)
\\=&\frac1p((1-a)A^2\gamma^2+2(1-a)AB\gamma+(1-a)B^2
\\&+(1+a)A^2+2(1+a)AB\gamma^{-1}+(1+a)B^2\gamma^{-2})
\\=&\frac1p((1+a)A^2+(1-a)B^2+3(1+a)AB\gamma^{-1}+3(1-a)AB\gamma)
\\=&\frac1p((1-a)p^2+a^2(1+a)+\\
&3(p-pa-a^2)(pa+a+a^2)((1-a)\gamma+(1+a)\gamma^{-1})) 
\end{align*}
\end{proof}

\begin{lemma}\label{lemintp}
Let $f\in C^2(I)$ such that $f(0)=f(1)=0$. Then for every $y\in I$, there
is $t_y\in I$ such that
\begin{equation}\label{eqintp}
f(y)=-\frac{y(1-y)}{2}f'(t_y).
\end{equation}
\end{lemma}

\begin{proof}(\ref{eqintp}) holds obviously for $y=0$ and $y=1$.
Fix $y\in (0,1)$.
Let $g(t)=y(1-y)f(t)-t(1-t)f(y)$. Then $g(0)=g(y)=g(1)=0$. There must be
$t_y\in I$ such that $g''(t_y)=0$, which just implies (\ref{eqintp}).
\end{proof}

Analogous to Wirsing's approach, Theorem \ref{thmeigen} has a corollary that can be used to compute $\lambda_p$
to any accuracy:
\begin{corollary}
For any $f\in C^1(I)$ with $f'>0$ and for every $n\in\mathbb{N}$,
$$\frac{(\scru^nf)(1)-(\scru^nf)(0)}{(\scru^{n-1}f)(1)-(\scru^{n-1}f)(0)}=-\lambda_p+O((\frac{\tau_p}{\lambda_p})^n)$$
(By (\ref{esttau2}), for $p\ge 2$ we have $\cfrac{\tau_p}{\lambda_p}<\cfrac{189}{198}$.)
\end{corollary}

\begin{proof}
By Theorem \ref{thmeigen},
$$(\scru^nf)'=\scrv^{n} f'=\lambda_p^n\scrl(f')\psi+\tau_p^n\osc(\frac{f'}{\psi})\kappa_n\psi.$$
Take the integrals on $I$. The result follows.
\end{proof}

Moreover, from Theorem \ref{thmeigen} we know that $1$ and $-\lambda_p$ are the first
two eigenvalues of $\scru$ (and also $\scrg_p$). Let us denote the $n$-th
eigenvalue of $\scrg_p$ by $\Lambda_p(n)$. A remarkable result by G. Alkauskas
\cite{GA}
shows that
\begin{equation}\label{gaest}
(-1)^{n-1}\Lambda_1(n)=\phi^{2n}+C\cdot\frac{\phi^{2n}}{\sqrt n}+d(n)\cdot\frac{\phi^{2n}}{n},
\end{equation}
where $\phi=\frac{\sqrt 5-1}{2}$, $C$ is a known constant and $d$ is a bounded function.
According to \cite{GA}, behind (\ref{gaest}) is the spectrum
of operator
$$\scrh_1f(x)=\frac1{(1+x)^2}f(\frac1{1+x}),$$
which is $\{(-1)^{n-1}\phi^{2n}:n\in\mathbb{N}\}$. It is natural to consider
$$\scrh_pf(x)=\frac p{(p+x)^2}f(\frac p{p+x}).$$
We conjecture that
\begin{conjecture} For every positive integer $p$,
$$\lim_{n\to\infty}\frac{\Lambda_p(n)}{(-1)^{n-1}p^{-n}\phi_p^{2n}}=1,$$
where $\phi_p=\cfrac{\sqrt{p^2+4p}-p}{2}$ is the positive root of the equation
$$x=\frac{p}{p+x}.$$
\end{conjecture}

\section*{Acknowledgments}
The author is supported by NSFC No. 11571387. 
%The author would like to thank
%Omri Sarig who first asked the author to think about how the Gauss measure
%can be found. 
 
%+Bibliography

%-Bibliography


\begin{thebibliography}{99}

\bibitem{GA} G. Alkauskas, \emph{Transfer operator for the Gauss' continued fraction map. I. Structure of the eigenvalues and trace formulas}, preprint,
2014.

\bibitem{IK} M. Iosifescu, C. Kraaikamp, Metrical Theory of Continued Fractions,
Kluwer Academic Publisher, Dordrecht, 2002.

%\bibitem{JT} W. Jones and W. Thron,
%\emph{Continued fractions. Analytic theory and applications},
%Encyclopedia of Mathematics and its Applications, 1980.

%\bibitem{Khin1} A. Khinchin,
%\emph{Metrische Kettenbruchprobleme}, Compositio Mathematica, 1 (1935), 359--382

%\bibitem{Khin2} A. Khinchin,
%Continued Fractions, 1964.

\bibitem{Kuzmin}R. O. Kuzmin, \emph{On a problem of Gauss},
Dokl. Akad. Nauk SSSR Ser. A (1928), 375--380.

%\bibitem{Levy} P. Levy, 
%\emph{Th\'eorie de l'addition des variables al\'eatoires}, Monographies des
%Probabilit\'es I, 1937.

\bibitem{Levy} P. L\'evy, Sur les lois de probabilit\'e dont d\'ependent
les quotients complets et incomplets d¡¯une fraction continue, Bull. Soc. Math. France
57 (1929), 178--194.


%\bibitem{CRN} C. Ryll-Nardzewski, 
%\emph{On the ergodic theorems II (Ergodic theory of continued fractions)}, %Studia Math. 12 (1951): 74--79.

%\bibitem{RS}
%A. M. Rockett, P. Sz\"usz, Continued Fractions, World Scientific, Singapore,
%1992.

\bibitem{Sun1} P. Sun,
\emph{Invariant measures for generalized Gauss transformations}, 
preprint,
%arXiv:1704.08281, 
2017.

\bibitem{Sun2} P. Sun,
\emph{A Generalization of Gauss-Kuzmin-L\'evy
Theorem},
preprint, 2017

\bibitem{Wir} E. Wirsing, 
\emph{On the Theorem of Gauss-Kuzmin-L\'evy and a Frobenius-Type Theorem for Function Spaces}. Acta Arith. 24 (1974), 507--528.

\end{thebibliography}
\end{document}